\documentclass{amsart}

\usepackage[english]{babel}
\usepackage{amsmath,amssymb,amsbsy,amsfonts,amsthm,latexsym,amsopn,amstext,amsxtra,euscript,amscd,paralist}
\usepackage{hyperref}
\usepackage{array}
\usepackage{ifthen}
\usepackage{url}

\newtheorem{thm}{Theorem}
\newtheorem{cor}{Corollary}
\newtheorem{lem}{Lemma}
\newtheorem{pro}{Proposition}

\newtheorem{rmq}{Remark}

\newcommand{\nbf}{\mathbf{n}}
\newcommand{\xbf}{\mathbf{x}}
\newcommand{\bsa}{\boldsymbol{\alpha}}
\newcommand{\KK}{\mathbb{K}}
\newcommand{\G}{\mathcal{G}}
\newcommand{\g}{\mathfrak{g}}
\newcommand{\zetan}{\tau}
\newcommand{\Y}{\beta}
\newcommand{\ri}{m_i}
\newcommand{\riff}{m_i}

\begin{document}

\title[Quantitative growth of multi-recurrence sequences]{Quantitative growth of multi-recurrences}

\author[C. Fuchs]{Clemens Fuchs}
\address{Clemens Fuchs\newline
	\indent University of Salzburg\newline
	\indent Department of Mathematics\newline
	\indent Hellbrunner Str. 34 \newline
	\indent A-5020 Salzburg, Austria.}
\email{clemens.fuchs@plus.ac.at}

\author[A. Noubissie]{Armand Noubissie}
\address{Armand Noubissie\newline
	\indent Graz University of Technology\newline
	\indent Institute of Analysis and Number Theory\newline
	\indent M\"{u}nzgrabenstrasse 36/II \newline
	\indent A-8010 Graz, Austria.}
\email{armand.noubissie@tugraz.at}
\thanks{\emph{Corresponding author}: A. Noubissie.}
\thanks{The first author was supported by the Austrian Science Fund (FWF), Grant-DOI 10.55776/I4406, the second author was funded by FWF, Grant-DOI 10.55776/P35322.}



\date{\today}
\begin{abstract}
In 1982, Schlickewei and Van der Poorten claimed that any multi-recurrence sequence has, essentially, maximal possible growth rate. Fourty years later, Fuchs and Heintze provided a non-effective proof of this statement. In this paper, we prove a quantitative version of that result by giving an explicit upper bound for the maximal possible growth rate of a multi-recurrence. Moreover, we also give a function field analogue of the result, answering a question posed by Fuchs and Heintze when proving a bound on the growth of multi-recurrences in number fields.
\end{abstract}

\maketitle

%%%%%%%%%%%%%%%%%%%%
\section{Introduction}\label{sec1}
%%%%%%%%%%%%%%%%%%%%
Let $K$ be a number field and let $\KK$ be an algebraic closure of $K$ embedded in $\mathbb{C}$. A sequence $G:\mathbb{N}\rightarrow K$ with $n\mapsto G_n$, abbreviated $(G_n)$, is a linear recurrence sequence (LRS) of order $l$ defined over $K$ if it is defined by the recursion relation
$$
G_{n+l}= a_1G_{n+l-1}+ a_2G_{n+l-2}+ \cdots + a_lG_{n}~(n \in \mathbb{N}),
$$
where $G_0,\ldots,G_{l-1}$ and the coefficients $a_1, \ldots, a_l$ are in $K$, and $l$ is minimal. It is well-known that we can write
\begin{equation} \label{eq:0}
G_n = \sum_{i=1}^{r} P_i(n) \alpha_i^n, 
 \end{equation} 
where $\alpha_1,\ldots,\alpha_r$ are the distinct roots of the characteristic polynomial $P=X^l-a_1X^{l-1}-\cdots-a_l$ of $(G_n)$ in $\KK$ and the coefficients of the polynomials $P_i$ belong to the splitting field $K(\alpha_1,\ldots,\alpha_r)$ of $P$ over $K$ for $i=1,\ldots,r$. The formula in \eqref{eq:0} is called power sum representation, or Binet-type formula, for $(G_n)$. One says that a LRS $(G_n)$ is non-degenerate if none of the ratios $\alpha_i/\alpha_j$ is a root of unity for any pair $(i,j)$ with $1\leq i<j\leq r$. Without loss of generality, let $\alpha_1$ be a root of $P$ with maximum modulus, $a_{ij}$ the coefficients of the polynomials $P_i$ and $\ri = \deg P_i \leq m$ with $m$ the maximum of all the multiplicities of $\alpha_i$ for $i=1,\ldots,r$. In this paper, we are concerned with the rate of growth of the non-degenerate LRS $(G_n)$. It is not difficult to see that there is an effectively computable constant $C_1$ such that, for all $n \geq 1, \vert G_n \vert < C_1 n^m \vert \alpha_1 \vert^n.$ In 1977, Loxton and Van der Poorten conjectured (cf. \cite[Conjecture 2]{LV}) that any non-degenerate LRS has essentially, the maximal possible growth rate, i.e., for any $\epsilon >0$ there is a effectively computable constant $C_2=C_2(\epsilon),$ such that if $ \vert G_n \vert < (\max_i\{ \vert \alpha_i \vert \})^{n(1-\epsilon)}$, then $n<C_2$. Using results of Schmidt \cite{S} and Evertse \cite{E1}, a complete non-effective (qualitative) proof of this conjecture was given by the first author and Heintze in \cite{FH} and, independently, by Karimov et al. in \cite{K}, as well as by Xiao in \cite{X} (cf. \cite[Lemma 4.1]{X}). Recently, the second author in \cite{NOW} gave an explicit upper bound on the number of solutions of that inequality based on the machinery used by Evertse to prove a quantitative version of Subspace theorem (cf. \cite{E}). We also refer to \cite[Lemma 5]{LM} for a related quantitative result also using \cite{E}.

If we allow more than one parameter, we can generalize \eqref{eq:0} to 
\begin{equation} \label{eq:1}
G(n_1, \ldots, n_t) = \sum_{i=1}^{r} P_i(n_1, \ldots, n_t)\alpha_{i1}^{n_1}\cdots \alpha_{it}^{n_t}, 
 \end{equation}
 where $r,t$ are positive integers, the $P_i$'s are polynomials in $t$ variables with coefficients given by the vector $\textbf{a}_i =(a_{i1}, \ldots, a_{i\ri})$, ordered lexicographically, for $i=1,\ldots,r$, and $n_1, \ldots, n_t$ are non-negative integers. The polynomial-exponential function $G: \mathbb{Z}^t \rightarrow K$ with $\nbf=(n_1,\ldots,n_t)\mapsto G(\nbf)$, abbreviated by $(G(\nbf))$, is called a multi-recurrence sequence (or multi-recurrence) and equation \eqref{eq:1} is called its power sum representation. Assume that $\alpha_{ij}$'s are algebraic integers over $K$ and the $P_i$'s have coefficients in $K(\alpha_{11},\alpha_{12},\ldots,\alpha_{rt})$. We put $\bsa_i^{\nbf} := \alpha_{i1}^{n_1} \cdots \alpha_{it}^{n_t}$ for $i=1,\ldots,r$. Without loss of generality, we assume that $\alpha_{11}$ has maximal modulus with $\vert \alpha_{11} \vert >1$ among the $\alpha_{ij}$'s. We also put $\vert \nbf \vert = \max_i\{\vert n_i \vert \}$. The following result was claimed by Schlickewei and Van der Poorten in \cite{VS} and was recently proved by the first author and Heintze in \cite{FH1}: 
{\it 
Fix $\epsilon >0$. Let $\mathcal{A}$ be the set of vectors $\nbf$ such that for any subset $I \subseteq \{1, \ldots, r\}$ with $1 \in I,$ we have $$\sum_{i \in I}^{} P_i(\nbf)\bsa_i^{\nbf} \neq 0.$$ Then, the inequality 
\begin{equation} \label{eq:2}
\vert G(\nbf)\vert < \vert P_1(\nbf) \bsa_1^{\nbf} \vert e^{-\epsilon (\vert n_1\vert+\cdots+\vert n_t\vert)}
 \end{equation} 
 has finitely many solutions in $\mathcal{A}$.
}
\begin{rmq}
 The condition on an element of the set $\{1, \ldots, r\}$ in the result of the first author and Heintze is really necessary, as without such restriction, inequality \eqref{eq:2} may have an infinite number of solutions. Indeed, consider the multi-recurrences $G(n_1,n_2,n_3,n_4) = 2^{n_1} - 2^{n_2} +3^{n_3} - 3^{n_4}$ with 
 $$\alpha_1= (2,1,1,1), \quad \alpha_2= (1,2,1,1), \quad \alpha_3= (1,1,3,1), \quad \alpha_4= (1,1,1,3).$$ Then, the inequality \eqref{eq:2} has infinitely many solutions, namely solutions with $n_1=n_2,n_3=n_4.$
 \end{rmq}

The definitions above make also sense for function fields. Let $K$ be a function field in one variable over $\mathbb{C}$ and let $L$ be a finite algebraic extension of $K$ of genus $\g$. Put $L^*=L\backslash\{0\}$. We shall work with vectors $\bsa=(\alpha_1, \ldots, \alpha_t) \in (L^*)^t$ with $t$ a positive integer, and we denote by $\bsa^\nbf := \alpha_1^{n_1}\cdots \alpha_t^{n_t}$ for an integer vector $\nbf=(n_1, \ldots,n_t)$. We say that two elements $\bsa_1, \bsa_2$ of $(L^*)^t$ are linearly independent modulo $\mathbb{C}$ if there is no non-zero integer vector $\nbf$ such that $(\bsa_1\bsa_2^{-1})^\nbf \in \mathbb{C}$. Let $(G_n)$ be a LRS defined over the function field $K$ with power sum representation \eqref{eq:0}, where for $i=1,\ldots,r$ the coefficients of the polynomials $P_i$ and the characteristic roots $\alpha_i$ are contained in $L$. Moreover, let $\mu$ be a valuation of $L$. It now follows trivially that there is an effectively computable constant $C_3$ such that $\mu(G_n)\geq {C}_3+n\cdot\min_i\{\mu(\alpha_i)\}$. The first author and Heintze proved in \cite{FH} that \emph{if $(G_n)$ is non-degenerate} (i.e. $\alpha_i/\alpha_j\notin\mathbb{C}$ resp. $\alpha_i,\alpha_j$ are linearly independent modulo $\mathbb{C}$ for all $1\leq i<j\leq n$)\emph{, then there is an effectively computable constant $C_4$, independent of $n$, such that for every sufficiently large $n$ the inequality $\mu(G_n)\leq C_4+n\cdot\min_i\{\mu(\alpha_i)\}$ holds.}

The purpose of this paper is two-fold. Firstly, we give a quantitative version of the result on multi-recurrences. For doing so, we again use Evertse's quantitative version of the Subspace theorem (see Section \ref{sec3}). This gives an explicit upper bound for the maximal possible growth rate of a multi-recurrence. Secondly, we give a function field analogue on the growth of multi-recurrences. This solves an open question posed in \cite{FH1}. For the proof of this result we use, as usual, the Brownawell-Masser inequality as well as Zannier's function field analogue of the Subspace theorem (again see Section \ref{sec3}). Along the way, we prove a function field version of a result on multiplicative independence, which uses ideas going back to Loxton and Van der Poorten (see Lemma \ref{lemg}). The results over number fields will be given in Section \ref{sec:nf}, the results over function fields in Section \ref{sec:ff}. We collect some preliminaries in Section \ref{sec3} before we give the proofs in the subsequent sections.

\section{Growth of multi-recurrences over number fields}\label{sec:nf}

Let $K$ be a number field and let $(G(\nbf))$ be a multi-recurrence defined over $K$. Let $d$ be the degree of $K$ over $\mathbb{Q}$, $S$ the set containing all the prime ideals above the $\alpha_{ij}$'s and all the Archimedean places over $K$, and let $s$ the cardinality of $S$. 
We assume that $\alpha_{11}$ has maximum modulus with $\vert \alpha_{11}\vert >1$ among $\alpha_{ij}'s$ and moreover $$ \vert \alpha_{11}\vert = \max_{i,j, \delta}   \vert \delta (\alpha_{ij} )\vert_v$$ where the maximum is taken over all elements $\delta$ of ${\rm Gal}(K/ \mathbb{Q})$.  We denote by $q$ the least common multiple of all the denominators of coefficients of the polynomials $P_i$ for $i=1,\ldots,r$. Put $$ B=\max_{\sigma, i,j} \{ \vert \sigma(qa_{ij})\vert \},$$ where the maximum is taken over all elements $\sigma$ of ${\rm Gal}(K/ \mathbb{Q})$ and the coefficients $a_{ij}$ of $qP_i$ for $j=1,\ldots,m_i,i=1,\ldots,r$. Let $\riff$ be the total degree of the polynomial $P_i$ for $i=1,\ldots,r$. For $x \in \mathbb{R}\backslash\{0\}$, we define the function $$ T(x)= \dfrac{20rd\left( \max_i\{\riff\} + \log (2^{\max_i\{\riff\}+t} B) \right)}{xt \log \vert \alpha_{11} \vert} $$ and 
 $$\zetan (x) = \max \biggl\{ 10, 2T(x) \log T(x) \biggl\}. $$ Let $\Gamma = \{1, \ldots,r\}$ and $\G$ be the subset of $\mathbb{Z}^t$ consisting of \textbf{z} with $\alpha_l^{\textbf{z}} = \alpha_k^{\textbf{z}}$ for some $l,k \in \Gamma$ with $l \neq k$. We denote by $$A' = \sum_{i \in \Gamma}^{} \binom{t+\ri}{t} \mbox{ and } A= \max \{t, A'\}.$$
Notice that $A = \max\{t,r\}$ if all the polynomials $P_1, \ldots, P_r$ are constants, and $A=A'$ otherwise. We write $$\mathcal{S}_i = \{ \textbf{z} \in \mathbb{Z}^t: P_i(\textbf{z}) = 0\}$$ for $i=1,\ldots,r$.
The main result of this section is the following

\begin{thm} \label{thm2}
Suppose $\G = \{\textbf{0}\}$. Then, for $\epsilon >0$, the set $\mathcal{M}_{\epsilon}$ of solutions of the inequality 
\begin{equation} \label{eq:2mr}
 \vert G(\nbf)\vert < \vert P_1(\nbf) \bsa_1^{\nbf} \vert \cdot \vert \alpha_{11} \vert ^{-\vert \nbf \vert t \epsilon} 
 \end{equation} 
 is contained in $\mathcal{C} \cup \mathcal{S}_1 \cup \cdots \cup \mathcal{S}_r$, where $\mathcal{C}$ has cardinality
 $$ \vert \mathcal{C} \vert \leq (\zetan(\epsilon/(2d)) +1)^t +\frac{r!}{e}\cdot 2^{35A^3}d^{6A^2}\cdot 2\cdot \left( 2^{60r^2}\cdot \left(\frac{10}{22d}\epsilon\right)^{-7r}\right)^s.$$
 \end{thm}
\begin{rmq}
\begin{itemize}
\item The index ``1" in the above theorem is not relevant in the sense that the quantity $\vert P_1(\nbf) \bsa_1^{\nbf} \vert$ can be replaced by any expression $\vert P_{i_0}(\nbf) \bsa_{i_0}^{\nbf} \vert$ with $i_0 \in \Gamma$.
\item If $K$ is a number field which does not contain the $\alpha_{ij}$'s, then Theorem \ref{thm2} still holds by replacing $\vert \alpha_{11}\vert$ by $\max_{i,j}\{ H(\alpha_{ij})\}$. (For the definition of the height function $H$ see Section \ref{sec3} below.)
 \end{itemize}
\end{rmq}
 \begin{rmq}
 If the polynomials $P_i$ are non-zero constants, then by replacing in the proof of Theorem \ref{thm2} Schmidt and Schlickewei's result (cf. \cite[Theorem 1]{SS}) by \cite[Theorem 1]{ESS}, it follows that, for $\epsilon >0$, the number of solutions of the inequality \eqref{eq:2mr} does not exceed $$(\zetan(\epsilon/(2d)) +1)^t + {\rm exp}((6(r-1))^{3(r-1)} (r(s-1)+1))\cdot \frac{2r!}{e}\cdot \left( 2^{60r^2}\cdot \left(\frac{10}{22d}\epsilon\right)^{-7r}\right)^s.$$ Hence we get a better bound than those of the theorem \ref{thm2}. 
 \end{rmq}

The proof of Theorem \ref{thm2} will be given in Section \ref{sec4}.

Let $\mathcal{A}$ consist of vectors \textbf{z} in $\mathbb{Z}^t$ such that for any subset $I \subseteq \{1, \ldots,r\}$ with $1 \in I$, we have $$\sum_{i \in I}^{} P_i(\textbf{z})\bsa_i^{\textbf{z}}\neq 0.$$ The corollary below gives us a quantitative version of \cite[Theorem 1]{FH1}and a generalization of \cite[Theorem 1]{SS}.

\begin{cor} \label{cor1}
Suppose $\G = \{\textbf{0}\}$. Then, for $\epsilon >0$, the number of elements of $\mathcal{A}$ solutions of the inequality \eqref{eq:2mr} does not exceed $$ 2^{r-1} \biggl( (\zetan(\epsilon / (2d)) +1)^t + \frac{r!}{e}2^{35A^3}d^{6A^2}\cdot 2 \cdot \left( 2^{60r^2}\cdot \left(\frac{10}{22d}\epsilon\right)^{-7r}\right)^s \biggl ).$$
\end{cor}
 
\begin{proof}
For any $\nbf \in \mathcal{A} $ there are at most $2^{r-1}$ multi-recurrences $(G(\nbf))$ with the properties $$G(\nbf)= \sum_{i \in I}^{} P_i(\nbf)\bsa_i^{\nbf}$$ for some subset $I \subseteq \{1, \ldots,r\}$ with $1 \in I $ and having no vanishing subsums, in particular $\nbf \notin \mathcal{S}_i$ for all $i \in I$. For each such multi-recurrence, it follows by Theorem \ref{thm2} that the number of elements of $\mathcal{A}$ solutions of the inequality 
\begin{equation} \label{eq:5}
 \vert G(\nbf)\vert < \vert P_1(\nbf) \bsa_1^{\nbf} \vert \cdot \vert \alpha_{11} \vert ^{-\vert \nbf \vert t \epsilon} 
\end{equation} 
does not exceed $$ (\zetan(\epsilon / (2d)) +1)^t + \frac{r!}{e}2^{35A^3}d^{6A^2}\cdot 2 \cdot \left( 2^{60r^2}\cdot \left(\frac{10}{22d}\epsilon\right)^{-7r}\right)^s.$$ Hence, putting all the cases together, we get the desired result.
\end{proof}
 
\section{Growth of multi-recurrence over function fields}\label{sec:ff}

Let $K$ be a function field in one variable over $\mathbb{C}$ and let $L$ be a finite extension of $K$. Let us recall the definitions of the discrete valuations on the field $\mathbb{C}(z)$ where $z$ is a transcendental element of $L$. For each element $a$ in $\mathbb{C}$, every non-zero element $Q$ of $\mathbb{C}(z)$ may be expanded as a formal Laurent series $$ { \displaystyle \sum_{n \geq m}} c_n (z-a)^n,$$ where $m\in\mathbb{Z},c_n \in \mathbb{C}$ for $n\geq m$ and $c_m \neq 0$. The valuation $\mu_a$ on $\mathbb{C}(z)$ is defined as $\mu_a(f)=m.$ Further, for $Q= f/g$ with $f, g \in \mathbb{C}[z],$ we put ${\deg}(Q) = {\deg}(f) - {\deg}(g)$, and $\mu_{\infty}(Q)= - {\deg}(Q)$. By definition $L$ is a finite extension $\mathbb{C}(z)$. Each valuation $\mu_a, \mu_{\infty}$ can be extended in at most $d=[L:\mathbb{C}(z)]$ ways to a discrete valuation on $L$ and we denote by $\mathcal{M}_L$ the set of such valuations. Notice that a valuation on $L$ is called finite if it extends $\mu_a$ for some $a \in \mathbb{C}$, and infinite if it extends $\mu_{\infty}$. Restricting the valuations in $\mathcal{M}_L$ to $K$, gives the discrete valuations $\mathcal{M}_K$ of $K$; again each valuations on $K$ extends to at most $[L:K]$ valuations on $L$.

We define the projective height of a non-zero vector $(x_1, \ldots, x_r)$ with coordinate in $L$ as usual by $$\mathcal{H}(x_1, \ldots,x_r) = -{ \displaystyle \sum_{\mu \in \mathcal{M}_L}} {\min}\{\mu(x_1), \ldots, \mu(x_r)\}.$$ For a single element $f \in L$, we define $$\mathcal{H}(f) = \mathcal{H}((1, f)) = -{ \displaystyle \sum_{\mu \in \mathcal{M}_L}} {\min}\{0, \mu(f)\}.$$ Our main theorem is the following

\begin{thm}\label{thm9}
Let $\bsa_i=(\alpha_{i1}, \ldots, \alpha_{it}) \in (L^*)^t$ with $i \in \{1, \ldots, r\}$ such that $\bsa_i,\bsa_j$ are linearly independent modulo $\mathbb{C}$ for each pair $(i,j)$ with $1 \leq i<j \leq r.$ Moreover, for every $i \in \{1, \ldots , r\}$ let $\pi_{i1}, \ldots, \pi_{ir_i}$ be $r_i$ linearly independent elements over $\mathbb{C}$. Then, for every vector $\nbf=(n_1, \ldots, n_t) \in \mathbb{Z}^t$ such that $$\{\pi_{il}\bsa_i^\nbf : l=1, \ldots, r_i, i=1, \ldots,r\}$$ is linearly dependent over $\mathbb{C}$, but no proper subset of this set is linearly dependent over $\mathbb{C}$, we have $$\vert \nbf \vert = {\max}_i \{\vert n_i \vert\} < C_5=C_5( \g, \pi_{il}, \bsa_i\mbox{ for }l\in\{1,\ldots,r_i\},i\in\{1,\ldots,r\}, \vert S\vert),$$ where $$C_5=(r+1)!\prod_{i=1}^r\mathcal{H}(\alpha_i) \left[ {\max}_{i,j,l,u} \left\{\mathcal{H}\left( \frac{\pi_{il}}{\pi_{ju}}\right)\right\} + \binom{{\displaystyle \sum_{i=1}^{r}} r_i-1}{2} (\vert S\vert + 2\g - 2) \right],$$ where
$$\alpha_l = \frac{\alpha_{il}}{\alpha_{jl}} \mbox{ for }l\in\{1,\ldots,t\},$$ $S$ is the finite set containing all the infinite places of $L$ and zeros and poles of $\alpha_{ip}$ and $\pi_{il}$ for $i=1, \ldots, r$; $p=1, \ldots, t$ and $l= 1, \ldots, r_i,$ and where the maximum is taken over all pairs $(i,j)$ with $1\leq i<j\leq r$ and all $l\in\{1,\ldots,r_i\},u\in\{1,\ldots,r_j\}$.
\end{thm}

Notice that our Theorem \ref{thm9} is a generalization of \cite[Theorem 2.1]{FP} and of \cite[Proposition 3]{FH}. Also notice that Theorem \ref{thm9} immediately implies and generalizes \cite[Theorem 1]{FH2}. As further consequence of this result, we get a function field analogue on the growth of multi-recurrences.

\begin{cor}\label{core}
Let $(G(\nbf))$ be a multi-recurrence defined over $K$ with power sum representation $$G(\nbf) = P_1(\nbf)\bsa_1^\nbf + P_2(\nbf)\bsa_2^\nbf + \cdots + P_r(\nbf)\bsa_r^\nbf,$$ where $\bsa_i= (\alpha_{i1}, \ldots, \alpha_{it}) \in (L^*)^t$ and all the coefficients of $P_i$ for $i=1,\ldots,r$ belong to $L$. Let $\mu$ be a valuation on $L$. Assume that $\bsa_i,\bsa_j$ are linearly independent modulo $\mathbb{C}$ for each pair $(i,j)$ with $1 \leq i<j \leq r.$ Then there are effectively computable constants $C_5, C_6$, independent of $n$, such that, for every integer vector $\nbf=(n_1, \ldots, n_t)$ with $P_1(\nbf)\cdots P_r(\nbf) \neq 0$, if $\vert \nbf \vert \geq C_5,$ then $$\mu(G(\nbf)) \leq C_6 + {\min}_i\{ n_1 \mu(\alpha_{i1}) +\cdots+ n_t \mu(\alpha_{it}) \}.$$
We may take $$C_5=(r+1)!\prod_{i=1}^r\mathcal{H}(\alpha_i) \left[ {\max}_{i,j,l,u} \left\{\mathcal{H}\left( \frac{\pi_{il}}{\pi_{ju}}\right)\right\} + \binom{{\displaystyle \sum_{i=1}^{r}} {\deg}P_i-1}{2} (\vert S\vert + 2\g - 2) \right]$$ and $$C_6 = {\max}_{i,l} \{\mu(\pi_{il})\}+ \binom{{\displaystyle \sum_{i=1}^{r}} {\deg}P_i-1}{2} (\vert S\vert + 2\g - 2),$$
where $S$ is a finite set of place of $L$ containing all zeros and poles of $\alpha_{ij}$ and $\pi_{il}$, $\mathbb{C}$-basis for the $\mathbb{C}$-space $V_i$ generated by the coefficients of $P_i$, as well as $\mu$ and the infinite places, and where the maximum is taken over all pairs $(i,j)$ with $1\leq i<j\leq r$ and $l\in\{1,\ldots,\dim V_i\},u\in\{1,\ldots,\dim V_j\}$.
\end{cor}

This result answers the open question in \cite{FH1} and is hence a generalization of function field analogue of the Loxton-Van der Poorten conjecture \cite{LV}. As further special case we get a quantitative version of \cite[Theorem 1]{FH}, which we record in the following corollary.

\begin{cor}\label{corf}
Let $(G_n)$ be a non degenerate linear recurrence sequence defined over $K$ with power sum representation $$G_n = P_1(n)\alpha_1^n + P_2(n)\alpha_2^n + \cdots + P_r(n)\alpha_r^n,$$ where for all $i=1,\ldots,r$ all $\alpha_i$ and all the coefficients of $P_i$ belong to $L=K(\alpha_1, \ldots, \alpha_r)$. Let $\mu$ be a valuation on $L$. Then there is an effectively computable constant $C_7$, independent of $n$ and the genus of $L$, and an effectively computable constant $C_8$, independent of $n$, such that, if $n \geq C_7,$ then $$\mu(G_n) \leq C_8 + n\cdot {\min}_i\{\mu(\alpha_{i})\}.$$
\end{cor}

The proof of Theorem \ref{thm9} will be given in Section \ref{sec5} and the proofs of the corollaries in Section \ref{sec6} and \ref{sec7} respectively. 

%#################################% 
\section{Preliminaries}\label{sec3}
%#################################%

In the section we collect some known results, which we will use in the proofs of our results. We divide the section according to preliminaries in the number field and the function field case respectively.

\subsection{Number fields: Evertse's quantitative Subspace theorem}
Let $K$ be a number field of degree $d$. Let $M_{K}$ be the collection of places of $K$. For $v \in M_K$, $x \in K$, we define $\vert x \vert_v$ as follows:
\begin{itemize}
\item $\vert x \vert_v = \vert \delta (x) \vert ^{1/d}$ if $v$ corresponds to the embedding $\delta: K \rightarrow \mathbb{R}$,\\
\item $\vert x \vert_v = \vert \delta (x) \vert^{2/d}$ if $v$ corresponds to the embedding $\delta: K \rightarrow \mathbb{C}$,\\
\item $\vert x \vert_v = (N(\mathfrak{P}))^{-{\rm ord}_{\mathfrak{P}}(x)/d}$ if $v$ corresponds to the prime ideal $\mathfrak{P}$ of $\mathcal{O}_K$ and ${\rm ord}_{\mathfrak{P}}(x)$ the exponent of $\mathfrak{P}$ in the decomposition of the ideal generated by $x$. 
 \end{itemize}
We call $v$ (real resp. complex) infinite if $v$ corresponds to an embedding (in $\mathbb{R}$ resp. $\mathbb{C}$). $v$ is called finite if $v$ corresponds to a prime ideal. The definitions are such that the product formula $$ \prod_{ v \in M_K}^{} \vert x \vert_v =1, \mbox{ for } x \in K^*=K \setminus \{0\}$$ holds. For a finite subset $S$ of cardinality $s$ containing the infinite $v$ of $M_K$, the ring of $S$-integers is defined as $$\mathcal{O}_S= \{x \in K: \vert x \vert_v \leq 1 \mbox{ for all } v\notin S\}$$ and its group of units is given by $$\mathcal{O}_S^*= \{x \in K: \vert x \vert_v = 1 \mbox{ for all } v\notin S\}.$$ For $v \in M_K$, the quantity $s(v)$ is given by
\begin{equation*}
s(v) = \left\{
\begin{aligned}1/d & \mbox{ if } v \mbox{ is real infinite},\\2/d & \mbox{ if } v \mbox{ is complex infinite},\\ 0 & \mbox{ if } v \mbox{ is finite}.
\end{aligned}
\right.
\end{equation*}
By the definition of $s(v)$, we get $$\sum_{v \in S}^{} s(v) = 1.$$ 
We define the absolute value of the vector $\xbf=(x_1, \ldots, x_m) \in K^m$ with $\xbf \neq 0$ by
 \begin{equation*}
\vert \xbf \vert_v = \left\{
\begin{aligned}
\left( \sum_{i=1}^{m} \vert x \vert_v^{2d} \right)^{s(v)/2} & \mbox{ if } v \mbox{ is real infinite}, \\
	 \left( \sum_{i=1}^{m} \vert x \vert_v^{d} \right)^{s(v)/2} & \mbox{ if } v \mbox{ is complex infinite},\\
\max \left\{ \vert x_1\vert_v, \ldots, \vert x_m \vert_v \right\} & \mbox{ if } v \mbox{ is finite}.
\end{aligned}
\right.
\end{equation*}

Now, the height of $\xbf$ is defined as follows: $$H(\xbf)= \prod_{ v \in M_K}^{} \vert \xbf \vert_v,$$ and by applying the product formula it follows that $H(\xbf)$ depends only on $\xbf$ and not on the choice of the number field $K$. The height has the following properties (see, e.g. (\cite[Lemma 2.1]{AN})):

\begin{lem}\label{leme2}
For $\eta,\gamma\in \KK^*$, we have
\begin{compactitem}
\item[a)] ${H}(\eta) \geq 1$ and ${H}(\eta) = {H}(\eta^{-1})$,
\item[b)] ${H}(\eta + \gamma) \leq 2{H}(\eta){H}(\gamma)$,
\item[c)] ${H}(\eta^n) ={H}(\eta)^{\vert n\vert},$ for any $n \in \mathbb{Z}$,
\item[d)] ${H}(\eta\gamma) \leq  H(\eta)H(\gamma)$.
\end{compactitem}
\end{lem}

We shall require the following quantitative version of Schmidt's Subspace theorem due to Evertse \cite{E}.

\begin{thm}[Subspace theorem]\label{thm1}
Let $\{L_{1v}, \ldots, L_{rv}\}~(v \in S)$ be a linearly independent set of linear forms in $r$ variables with coefficients in $K$ such that $H(L_{iv}) \leq H$ for $i\in\{1,\ldots,r\},v\in S$. Let $0 <\epsilon < 1.$ Consider the inequality 
 \begin{equation} \label{eq:3}
\prod_{ v \in S}^{} \prod_{i=1}^{r} \dfrac{\vert L_{iv}(\xbf)\vert_v}{\vert \xbf \vert_v} < \prod_{ v \in S}^{} \vert {\rm det}(L_{1v}, \ldots, L_{rv}) \vert_v H(\xbf)^{-r-\epsilon} \mbox{ with } \xbf \in K^r.
 \end{equation}
 There are proper linear subspaces $T_1, \ldots, T_{t_1}$ of $K^r,$ with $$t_1 < \left( 2^{60r^2}\epsilon^{-7r}\right)^s$$ such that every solution $\xbf \in K^r$ of inequality \eqref{eq:3} with $H(\xbf) \geq H$ belongs to $$T_1 \cup \cdots \cup T_{t_1}.$$
 \end{thm}
 
 \subsection{Function fields}

Let $K$ be a function field in one variable over $\mathbb{C}$ and let $L$ be a finite extension of $K$. First, we notice that the height function defined above satisfies some basic properties that are listed in the next lemma (proven, e.g., in \cite{FKK}).

\begin{lem}\label{leme}
For $f, g \in L^*$, we have
\begin{compactitem}
\item[a)] $\mathcal{H}(f) \geq 0$ and $\mathcal{H}(f) = \mathcal{H}(f^{-1})$,
\item[b)] $\mathcal{H}(f) - \mathcal{H}(g) \leq \mathcal{H}(fg) \leq \mathcal{H}(f) + \mathcal{H}(g)$,
\item[c)] $\mathcal{H}(f^n) = \vert n \vert \mathcal{H}(f),$ for any $n \in \mathbb{Z}$,
\item[d)] $\mathcal{H}(f) = 0 \Longleftrightarrow f \in \mathbb{C}^*$.
\end{compactitem}
\end{lem}

Let $S$ be a finite set of valuations of $L$ containing all infinite ones. Then $f \in L$ is called an $S$-unit if $\mu(f)=0$ for all $\mu \notin S$. Now, we state the following result due to Brownawell and Masser \cite{BM}, which is a generalization of a result due to Mason \cite{M}.

\begin{pro}\label{lemf}
Let $u_1, \ldots, u_n\in L^*~(n \geq 3)$ be such that $$u_1+u_2 + \cdots + u_n =0,$$ but no proper non-empty subset of the $u_i's$ is made up of elements linearly dependent over $\mathbb{C}$. Then $$\mathcal{H}(u_1, \ldots, u_n) \leq \dfrac{(n-1)(n-2)}{2} (\vert S\vert + 2\g -2),$$ where $S$ is the set of places of $L$, where $u_i$ is not unit, and $\g$ is the genus of $L$.
\end{pro}

Moreover, we will need the function field analogue of the Subspace theorem due to Zannier \cite{Z}.

\begin{pro}\label{lemh}
Let $L$ be a function field of genus $\g$. Let $\rho_1,\ldots, \rho_n \in L$ be linearly independent over $\mathbb{C}$, let $r \in \{0, \ldots, n\}$, and let $\mu$ a place of $L$. Let $S$ be a finite set of places of $L$ containing all the poles of $\rho_1,\ldots, \rho_n$ and all the zero of $\rho_1,\ldots, \rho_r.$ Put $$\delta = {\displaystyle \sum_{i=1}^{n}} \rho_i.$$ Then $${\displaystyle \sum_{\mu \in S}^{}}(\mu(\delta) - \min_i \{\mu(\rho_i)\}) \leq {\displaystyle \sum_{i=r+1}^{n}} \mathcal{H}(\rho_i) + \dfrac{(n-1)(n-2)}{2} (\vert S\vert + 2\g - 2).$$
\end{pro}

%#################################% 
\section{Proof of Theorem \ref{thm2}}\label{sec4}
%#################################%

 If $r=1$, then Theorem \ref{thm2} follows easily. Thus, we assume $r>1$. We denote by $id$ the embedding over $K$ corresponding to the identity. For $v \in S \setminus \{id\}$, we define $r$ linear forms $L_{1v}, \ldots, L_{rv}$ in $r$ variables $\xbf=(x_1, \ldots, x_r)$ as follows: $L_{iv}(\xbf)=x_i$ for $i=1,\ldots,r$. For $v=id,$ we define $L_{1v} (\xbf) = x_1+\cdots + x_r$ and $L_{iv}(\xbf)=x_i$ for $i=2,\ldots,r.$ We denote by $q$ the least common multiple of all the denominators of coefficients of polynomials $P_i's$ and put $$\mathcal{N} := \{(qP_1(\nbf)\bsa_1^{\nbf}, \ldots, qP_r(\nbf)\bsa_r^{\nbf}) : n \in \mathbb{N}\}.$$ Put $\xbf_{\nbf} = (qP_1(\nbf)\bsa_1^{\nbf}, \ldots, qP_r(\nbf)\bsa_r^{\nbf})$ and, without loss of generality, we assume $q=1$ since, the case $q \neq 1$ follows easily by replacing the sequence $G(\nbf)$ by $qG(\nbf)$. We prove the following result, which is the main ingredient in the proof of our theorem \ref{thm2}.

\begin{lem} \label{lem1}
For each $\epsilon >0$, the set of $\xbf_{\nbf} \in \mathcal{N} $ with $\vert \nbf \vert > \zetan(\epsilon)$ satisfying the inequality 
\begin{equation} \label{eq:7}
\prod_{ v \in S}^{} \prod_{i=1}^{r} \vert L_{iv}(\xbf_{\nbf})\vert_v < \vert \alpha_{11} \vert ^{-\vert \nbf \vert t \epsilon} 
\end{equation} 
is contained in $C_9$ many proper subspaces in $K^r$, which does not exceed $$ 2\cdot \left( 2^{60r^2}\cdot \left(\frac{10}{11}\epsilon\right)^{-7r}\right)^s.$$
\end{lem}

\begin{proof}
Assume $\vert \nbf\vert > \zetan(\epsilon),$ and $\xbf_\nbf$ as in the lemma satisfying inequality \eqref{eq:7} such that $L_{iv}(\xbf_n) \neq 0$ for $v \in S, i=1, \ldots,r.$ It is clear that the linear forms $\{L_{1v}, \ldots, L_{rv}\}$ are linearly independent over $K$, and for all $v \in S$ it holds $\vert {\rm det}(L_{1v},$ $\ldots, L_{rv}) \vert_v =1.$ Since the coefficients of the polynomials $P_i$ are algebraic integers and $\alpha_{il} \in \mathcal{O}_S^*$, it follows that $\xbf_\nbf\in\mathcal{O}_S^r$. Thus, $\vert \xbf_{\nbf}\vert_v \leq 1$ for all $v \notin S$ and therefore
\begin{equation} \label{eq:3d}
\dfrac{1}{ \prod_{ v \notin S}^{}\vert \xbf_{\nbf}\vert_v} \geq 1.
\end{equation} 
By using (\cite[Lemma 2.4]{AN}) and the fact that $\vert \nbf\vert > \tau(\epsilon)$, we obtain $$ H(\xbf_\nbf) < \vert \alpha_{11}\vert^{\frac{11t\vert \nbf \vert}{10}}.$$ 
 Namely,  $\vert \xbf_{\nbf}\vert_v \leq 1$ for all $v \notin M_{\infty}$, therefore $ H(\xbf_\nbf) \leq  \prod_{ v \in M_K}^{} \vert \xbf_{\nbf}\vert_v$. Now we fix $v \in M_{\infty}$. By using the fact that $\vert \cdot \vert_v^{1/s(v)}$ is the usual absolute value, we get
 $$
\begin{array}{lcl}
\vert \xbf_{\nbf}\vert_v    &= &  \vert (P_1(\nbf)\bsa_1^{\nbf}, \ldots, P_r(\nbf)\bsa_r^{\nbf}) \vert_v \\\\
   
                                    &\leq& r \max_i   \vert P_i(\nbf)\bsa_i^{\nbf} \vert_v \\\\
                
                                   &\leq& \left(r \dfrac{\max_i   \vert P_i(\nbf) \vert_v }{\vert \alpha_{11} \vert ^{nts(v)/10}}\right) \cdot   (\vert \alpha_{11} \vert )^{11nts(v)/10} \\\\         
                                   
                                    &\leq&  (\vert \alpha_{11} \vert )^{11nts(v)/10},
 \end{array} 
 $$
  where, for the last inequality, we used the fact that $\vert \alpha_{11} \vert = \max_{i,j, \delta}   \vert \delta (\alpha_{ij} )\vert_v$ and $n> \tau(\epsilon)$ and (\cite[Lemma 2.4]{AN}). Since $$\sum_{v \in M_{\infty}}^{} s(v) = 1,$$ it follows that  $H(\xbf_\nbf) \leq (\vert \alpha_{11} \vert )^{11nt/10}.$
By combining this inequality, relation \eqref{eq:3d} and the fact that $\xbf_n$ is a solution of the inequality \eqref{eq:7}, we deduce
$$\prod_{ v \in S}^{} \prod_{i=1}^{r} \vert L_{iv}(\xbf_{\nbf})\vert_v < \left( \dfrac{1}{H(\xbf_n)} \right)^{\epsilon'} \left( \dfrac{1}{ \prod_{ v \notin S}^{}\vert \xbf_{\nbf}\vert_v}\right)^r,$$ where $\epsilon' = \frac{10}{11}\epsilon$. Hence, by applying Theorem \ref{thm1} with $H=1,$ it follows that $\xbf_n$ belongs to one of finitely many proper linear subspaces $T_l$ with $$l\leq t_1 \leq \left( 2^{60r^2}\cdot \left(\frac{10}{11}\epsilon\right)^{-7r}\right)^s.$$ So, the set of $\xbf_{\nbf} \in \mathcal{N} $ with $\vert \nbf \vert > \zetan(\epsilon)$ satisfying inequality \eqref{eq:7} is contained in $C_9$ many proper subspaces in $K^r$ with $$C_9  \leq sr + \left( 2^{60r^2}\cdot \left(\frac{10}{11}\epsilon\right)^{-7r}\right)^s < 2\cdot \left( 2^{60r^2}\cdot \left(\frac{10}{11}\epsilon\right)^{-7r}\right)^s,$$ which gives us the desired result.
\end{proof}

Let $\nbf$ be a solution of \eqref{eq:2} with $ \vert \nbf \vert > \zetan(\epsilon/(2d))$ and $\xbf_{\nbf} \in \mathcal{N}$. We want to show that $\xbf_{\nbf}$ is a solution of inequality \eqref{eq:7}. Fixing $v \in S$ and $i \in \Gamma$, we have 
\begin{equation} \label{eq:10}
\vert P_i(\nbf)\vert_v \leq ((2^{\riff+t})\vert \nbf \vert ^{\riff})^{s(v)}B^{s(v)} \leq ((2^{\max_i\{\riff\}+t})\vert \nbf \vert ^{\max_i\{\riff\}})^{s(v)}B^{s(v)}.
 \end{equation} 
By the relation \eqref{eq:10}, the product formula, and the fact that $\alpha_{il} \in \mathcal{O}_S^*$, we obtain 
\begin{equation} \label{eq:4}
\prod_{ v \in S}^{} \prod_{i=1}^{r} \vert P_i(\nbf) \bsa_i^{\nbf} \vert_v \leq \left( (2^{\max_i\{\riff\}+t})\vert \nbf \vert ^{\max_i\{\riff\}} B\right)^r.
 \end{equation}
 Hence,
\begin{equation*}\begin{split}
\prod_{ v \in S}^{} \prod_{i=1}^{r} \vert L_{iv}(\xbf_{\nbf})\vert_v
&= \dfrac{\vert G({\nbf}) \vert_{id}}{ \vert P_1(\nbf)\bsa_1^{\nbf} \vert_{id}} \cdot \prod_{ v \in S}^{} \prod_{i=1}^{r} \vert P_i(\nbf) \bsa_i^{\nbf} \vert_v \\
&\leq \left( (2^{\max_i\{\riff\}+t})\vert \nbf \vert ^{\max_i\{\riff\}} B\right)^r \cdot \vert \alpha_{11} \vert^{-\vert \nbf \vert s(id) t \epsilon}\\
&\leq \vert\alpha_{11} \vert^{-\vert \nbf \vert t \epsilon / (2d)},
\end{split}
\end{equation*}
where for the first inequality, we used \eqref{eq:4} and the fact that $\nbf$ is solution of inequality \eqref{eq:2mr}, and for the last one, \cite[Lemma 2.4]{AN} and the fact that $\nbf > \zetan(\epsilon /(2d))$. Applying Lemma \ref{lem1} with $\epsilon$ replaced by $\epsilon /2d$ we deduce that there are at most $$2\cdot \left( 2^{60r^2}\cdot \left(\frac{10}{22d}\epsilon\right)^{-7r}\right)^s$$ polynomial-exponential Diophantine equations of the shape 
\begin{equation} \label{eq:11}
\sum_{i \in I}^{} c_iP_i(\nbf)\bsa_i^{\nbf}=0
\end{equation}
with $(c_1, \ldots, c_r) \in K^r / \{\textbf{0}\}$ and $I \subseteq \Gamma$ such that $\nbf$ is solution of at least one of them. Let us consider a polynomial-exponential equation of the shape \eqref{eq:11}. Let $\mathcal{P}$ be a partition of the set $I$. The set $\lambda \subseteq I$ occurring in the partition $\mathcal{P}$ will be considered elements of $\mathcal{P}$. Given $\mathcal{P}$, consider the system of equations
\begin{equation} \label{eq:12}
\sum_{i \in \lambda}^{} c_iP_i(\nbf)\bsa_i^{\nbf}=0~(\lambda \in \mathcal{P}).
\end{equation} 
A solution $\nbf$ of \eqref{eq:12} is called $\mathcal{P}$-degenerate if a subsum of one of the equations of the system vanishes. Otherwise, we will say $\nbf$ is $\mathcal{P}$-non-degenerate. Let $M(\mathcal{P})$ be the set consisting of $\mathcal{P}$-non-degenerate solutions of the system \eqref{eq:12}. It is clear that every solution of \eqref{eq:11} lies in $M(\mathcal{P})$ for some partition $\mathcal{P}$. By \cite[Lemma 2.5]{L}, the set of partitions of $I$ which only contain subsets with at least two elements has a cardinality, which does not exceed $r! / e.$ Hence, by using the fact that $\G = \{\textbf{0}\}$ and \cite[Theorem 1]{SS}, it follows that the number of solutions of equation \eqref{eq:11} does not exceed $$\frac{r!}{e} 2^{35A^3}d^{6A^2}.$$ By combining this bound with the upper bound above on the number of polynomial-exponential Diophantine equations of the shape \eqref{eq:11}, we get the desired result.
\hfill$\square$

\section{Proof of Theorem \ref{thm9}}\label{sec5}

For the proof of the theorem, we will need the following lemma, which is a function field analogue over a function field of a result on multiplicative dependence over number field.

\begin{lem}\label{lemg}
Let $L$ be a function field in one variable defined over $\mathbb{C}$, and let $\alpha_1, \ldots, \alpha_r$ be non-zero element of $L$ which are multiplicative independent modulo $\mathbb{C}$, and let $\alpha_0 \in L$. If there are non-zero integers $k_1, \ldots, k_r$ such that $$\alpha_1^{k_1}\alpha_2^{k_2} \ldots \alpha_r^{k_r} = \alpha_0,$$ then 
$$\vert k_i \vert \leq 
\dfrac{(r+1)!}{\mathcal{H}(\alpha_i)}\mathcal{H}(\alpha_1)\cdots \mathcal{H}(\alpha_r)\mathcal{H}(\alpha_0) $$ for all $i \in \{1, \ldots, r\}$.
\end{lem}
Notice that this result is a generalization of \cite[Lemma 7]{FH2}. 
\begin{proof}
We follow closely the proof of \cite[Lemma 7.5.1]{EG}. First, we prove there is $(l_0, \ldots, l_r) \in \mathbb{Z}^{r+1}$ such that 
$$\alpha_1^{l_1}\alpha_2^{l_2} \ldots \alpha_r^{l_r} \alpha_0^{l_0} \in \mathbb{C}$$ and $$\vert l_i \vert \leq 
\dfrac{(r+1)!}{\mathcal{H}(\alpha_i)}\mathcal{H}(\alpha_1)\cdots \mathcal{H}(\alpha_r)\mathcal{H}(\alpha_0)$$ for all $i \in \{0, \ldots, r\}.$ By assumption, there is integers vector $(b_0, \ldots,b_r)$, with coordinate not zero, such that 
\begin{equation} \label{eq:13}
\alpha_1^{b_1}\alpha_2^{b_2} \cdots \alpha_s^{b_r} \alpha_0^{b_0} \in \mathbb{C}.
\end{equation}
Without loss of generality, we may assume $$b_0 \mathcal{H}(\alpha_0) \geq b_i \mathcal{H}(\alpha_i) \mbox{ for } i= 1,2, \ldots, r.$$ It is not difficult to see that the vector $(b_0, \ldots,b_r)$ is unique up to a scalar. Indeed, for any $(b_0', \ldots, b_r')$ satisfying the relation \eqref{eq:13}, we have 
 $$\alpha_1^{b_0b_1'-b_0'b_1}\alpha_2^{b_0b_2'-b_0'b_2} \cdots \alpha_r^{b_0b_r'-b_0'b_r} \in \mathbb{C}.$$ Since $\alpha_1, \ldots, \alpha_r$ are linearly independent modulo $\mathbb{C}$, it follows that $b_i'= \frac{b_0'}{b_0} b_i$ for $i \in \{0,1,2, \ldots, r\}$. So we are done. We set $\psi_n = ((n+1)/(n+3))$. Clearly, $\psi_n$ is a non-negative sequence which converge to $1.$ Let $\mathcal{M}_n$ be the set consist of the point $(x_0, \ldots, x_r) \in \mathbb{R}^{r+1}$ such that $${ \displaystyle \sum_{i=1}^{r}} \mathcal{H}(\alpha_i) \left| x_i - \frac{b_i}{b_0}x_0 \right|\leq \psi_n \mbox{ and } \vert x_0 \vert \leq (r+1)! \psi_n^{-r}\mathcal{H}(\alpha_1)\cdots \mathcal{H}(\alpha_r) =: C_{10}.$$
Clearly, $\mathcal{M}_n$ is a compact symmetric convex body. Moreover, $\rho \mathcal{O} \subseteq \mathcal{M}_n$, where $\mathcal{O}$ is the octahedron, consisting of the points $(t_0, \ldots, t_r) \in \mathbb{R}^{r+1}$ with $\vert t_0 \vert + \cdots + \vert t_r \vert \leq 1$ and $\rho$ the matrix defined as
$$\rho^{-1} = \begin{pmatrix}
C_{10}^{-1} & 0 & \cdots & 0 \\
-\mathcal{H}(\alpha_1)\psi_n^{-1}\frac{b_1}{b_0} & \mathcal{H}(\alpha_1)\psi_n^{-1} & \cdots & 0 \\
\vdots & \vdots & \ddots & \vdots\\
-\mathcal{H}(\alpha_r)\psi_n^{-1}\frac{b_r}{b_0} & 0 & \cdots & \mathcal{H}(\alpha_r)\psi_n^{-1} \\
\end{pmatrix}.$$
It is well known that the Lebesgue measure of $\mathcal{O}$ is $\lambda(\mathcal{O})= \frac{2^{r+1}}{(r+1)!}$. Therefore, we infer
$$\lambda(\mathcal{M}_n) \geq \lambda(\rho \mathcal{O})= \vert {\rm det}(\rho)\vert \cdot \lambda(\mathcal{O}) = 2^{r+1}.$$ Hence, we deduce by Minkowski's convex body theorem that $\mathcal{M}_n$ contains a non-zero integer point denoted $I_n$ for every $n$. Since $\psi_n$ converges to $1,$ it follows that there exists an integer vector $(l_0, \ldots, l_r)$ such that 
 \begin{equation} \label{eq:14}
{ \displaystyle \sum_{i=1}^{r}} \mathcal{H}(\alpha_i) \left| l_i - \frac{b_i}{b_0}l_0 \right|< 1 \mbox{ and } \vert l_0 \vert \leq (r+1)!\mathcal{H}(\alpha_1)\cdots \mathcal{H}(\alpha_r).
\end{equation}
Let us consider $M$ be the splitting field over $L$ of the polynomial $(X^{b_0}- \alpha_1) \cdots (X^{b_0}- \alpha_r)$. We denote the height on $M$ by $\mathcal{H}_M$ and the height on $L$ by $\mathcal{H}_L$, where $\mathcal{H}=\mathcal{H}_L$. Observe that $\mathcal{H}_M(f)=[M:L]\mathcal{H}_L(f)$ for all $f\in L$. Let $\gamma_0, \ldots, \gamma_r \in M$ be such that $\gamma_i^{b_0}= \alpha_i$ for all $i= 0, \ldots, r$. From the relation \eqref{eq:13}, we deduce that $\theta:= \gamma_0^{b_0} \cdots \gamma_r^{b_r} \in \mathbb{C}$. From Lemma \ref{leme}, we obtain (notice $\mathcal{H}_L(\alpha_i) = \mathcal{H}(\alpha_i)$) that
\begin{equation*}\begin{split}
[M : L]\mathcal{H}_L(\alpha_0^{l_0}\cdots \alpha_r^{l_r})
&=\mathcal{H}_M(\alpha_0^{l_0}\cdots \alpha_r^{l_r}) = \mathcal{H}_M(\alpha_0^{l_0}\cdots \alpha_r^{l_r}\theta^{-l_0}) \\
&= \mathcal{H}_M\left(\gamma_1^{l_1b_0-l_0b_1}\cdots \gamma_r^{l_rb_0-l_0b_r}\right)\\
&\leq { \displaystyle \sum_{i=1}^{r}} \vert b_0 \vert \mathcal{H}_M(\gamma_i) \left| l_i - \frac{b_i}{b_0}l_0 \right| \\
&\leq { \displaystyle \sum_{i=1}^{r}} \mathcal{H}_M(\alpha_i) \left| l_i - \frac{b_i}{b_0}l_0 \right| < [M : L],
\end{split}\end{equation*}
where for the last inequality we have used relation \eqref{eq:14}. Therefore, $$\mathcal{H}_L(\alpha_0^{l_0}\cdots \alpha_r^{l_r}) <1,$$ which means $\alpha_0^{l_0}\cdots \alpha_r^{l_r} \in \mathbb{C}.$ Since the vector $(b_0, \ldots,b_r)$ is unique up to a scalar and $$b_0 \mathcal{H}(\alpha_0) \geq b_i \mathcal{H}(\alpha_i) \mbox{ for } i= 1,2, \ldots, r,$$ it follows that 
$$\vert l_i \vert = \left| \frac{b_i}{b_0} \right|\cdot \vert l_0 \vert \leq 
\dfrac{(r+1)!}{\mathcal{H}(\alpha_i)}\mathcal{H}(\alpha_1)\cdots \mathcal{H}(\alpha_r)\mathcal{H}(\alpha_0)$$ for all $i \in\{0, \ldots, r\}$, where for the last inequality we utilized relation \eqref{eq:14}. We have proven that there is a non-zero integer vector $(l_0, \ldots, l_r)$ with $l_0\neq 0$ such that $\alpha_0^{l_0}\cdots \alpha_r^{l_r} \in \mathbb{C}$ and $$\vert l_i \vert \leq 
\dfrac{(r+1)!}{\mathcal{H}(\alpha_i)}\mathcal{H}(\alpha_1)\cdots \mathcal{H}(\alpha_r)\mathcal{H}(\alpha_0)$$ for all $i \in\{0, \ldots, r\}$. By assumption, one has $\alpha_0 = \alpha_1^{k_1}\cdots \alpha_r^{k_r}$. Using the fact $(l_0, \ldots, l_r)$ is uniquely determined up the scalar, we infer
$$\vert k_i \vert = \dfrac{\vert l_i \vert }{\vert l_0 \vert } \leq \dfrac{(r+1)!}{\mathcal{H}(\alpha_i)}\mathcal{H}(\alpha_1)\cdots \mathcal{H}(\alpha_r)\mathcal{H}(\alpha_0)$$ for all $i \in\{1, \ldots, r\}$. There the proof of the lemma is completed.
\end{proof}

Now we give a proof of our Theorem \ref{thm9}. By hypothesis, the set 
$$\Xi = \{\pi_{il}\bsa_i^\nbf: i= 1,2, \ldots, r, l=1, \ldots, r_i\}$$ is linearly dependent over $\mathbb{C}$. Then there is $c_{il} \in \mathbb{C}$, non all zero, such that
\begin{equation} \label{eq:15}
{ \displaystyle \sum_{i=1}^{r}\sum_{l=1}^{r_i}}c_{il}\pi_{il}\bsa_i^\nbf = 0.
\end{equation}
By using the fact that $\{\pi_{i1}, \ldots, \pi_{ir_i}\}$ is linearly independent over $\mathbb{C}$, we infer $r>1$. If $r=2$ and $r_1 = r_2 = 1$, then the relation \eqref{eq:15} becomes $$c_{11}\pi_{11}\bsa_1^\nbf + c_{21}\pi_{21}\bsa_2^\nbf = 0$$ and from Lemma \ref{leme}, we obtain $$\mathcal{H}((\bsa_1^{-1}\bsa_2)^\nbf) \leq \mathcal{H}(\pi_{11}) + \mathcal{H}(\pi_{21}).$$
Now, we assume that \eqref{eq:15} consists of at least three terms. By applying Proposition \ref{lemf} to the relation \eqref{eq:15}, we get that for every $(n_1, \ldots, n_r) \in \mathbb{Z}^{r}$ for which relation \eqref{eq:15} holds, but no proper subset of $\Xi$ is linearly dependent over $\mathbb{C}$, that
\begin{equation} \label{eq:16}
\mathcal{H}(\Xi) \leq \dfrac{(d-1)(d-2)}{2} (\vert S\vert + 2\g - 2).
\end{equation}
Here, $S$ is the finite set of valuations of $L$ containing all the infinite places of $L$ and zeros and poles of $\alpha_{ij}$ and $\pi_{il}$ for $i=1, \ldots, r, l=1,\ldots,r_i$ and $j= 1, \ldots, t$. Observe also that we have set $$d= {\displaystyle \sum_{i=1}^{r}} r_i.$$ It easy to see that 
$$\mathcal{H}(\Xi) \geq {\displaystyle {\max}}\left\{ \mathcal{H}\left( \frac{\pi_{il}}{\pi_{ju}} (\bsa_i\bsa_j^{-1})^\nbf \right)\right\}.$$ Together with the relation \eqref{eq:16}, we infer $${\max}\left\{ \mathcal{H}\left( (\bsa_i\bsa_j^{-1})^\nbf \right) \right\}\leq {\max}\left\{ \mathcal{H}\left( \frac{\pi_{il}}{\pi_{ju}}\right)\right\} + \dfrac{(d-1)(d-2)}{2} (\vert S\vert + 2\g - 2).$$
Notice that for all pairs $(i,j)$ with $1\leq i<j\leq r$ we have $$(\bsa_i\bsa_j^{-1})^\nbf = \left( \frac{\alpha_{i1}}{\alpha_{j1}}\right)^{n_1}\cdots \left( \frac{\alpha_{ir}}{\alpha_{jr}}\right)^{n_r}.$$ Since $\bsa_i, \bsa_j$ were assumed to be multiplicatively independent modulo $\mathbb{C}$, by applying Lemma \ref{lemg} with $\alpha_0 = ( \bsa_i\bsa_j^{-1})^\nbf$ and $$\alpha_l = \frac{\alpha_{il}}{\alpha_{jl}} \mbox{ for } l \in \{1,2, \ldots, r\},$$ one deduces that, for any $i,$
$$\vert n_i \vert \leq \dfrac{(r+1)!}{\mathcal{H}(\alpha_i)}\mathcal{H}(\alpha_1)\cdots \mathcal{H}(\alpha_r) \left[ {\max}\left\{ \mathcal{H}\left( \frac{\pi_{il}}{\pi_{ju}}\right)\right\} + \dfrac{(d-1)(d-2)}{2} (\vert S\vert + 2\g - 2) \right]. $$
Therefore, ${\max}\{\vert n_i \vert\}$ is bounded by $$(r+1)!\mathcal{H}(\alpha_1)\cdots \mathcal{H}(\alpha_r) \left[ {\max}\left\{ \mathcal{H}\left( \frac{\pi_{il}}{\pi_{ju}}\right)\right\} + \dfrac{(d-1)(d-2)}{2} (\vert S\vert + 2\g - 2) \right]$$
and the desired result is obtained.
\hfill$\square$

\section{Proof of Corollary \ref{core}}\label{sec6}

We fix $\nbf=(n_1, \ldots, n_t)$ and write $$P_i(\nbf) = {\displaystyle \sum_{\vert l_1\vert + \cdots + \vert l_t \vert \leq {\deg}P_i}^{}} \pi_{i,l_1,\ldots, l_t}Q_{i,l_1,\ldots, l_t}(\nbf),$$ where $\riff(=r_i)={\deg}P_i$ is the total degree of $P_i$, and for a fixed $i$, the set $$\{\pi_{i,l_1,\ldots, l_t}: \vert l_1\vert + \cdots + \vert l_t \vert \leq \riff\}$$ is linearly independent over $\mathbb{C}$, and the $Q_{i,l_1,\ldots, l_t} \in \mathbb{C}[X_1, \cdots, X_t].$ We assume that $$P_1(\nbf)\cdots P_r(\nbf) \neq 0.$$ Then there is $i,l_1,\ldots, l_t$ such that $Q_{i,l_1,\ldots, l_t}(\nbf) \neq 0$. Put 
$$\Xi = \{\pi_{i,l_1,\ldots, l_t}\bsa_i^\nbf: i= 1,2, \ldots, r, \vert l_1\vert + \cdots + \vert l_t \vert \leq \riff \}.$$ Assume that the vector $\nbf$ is such that $\vert\nbf\vert={\max}\{ \vert n_i \vert\} \geq C_5$, where 
 $$C_5=(r+1)!\mathcal{H}(\alpha_1)\cdots \mathcal{H}(\alpha_r) \left[ {\max}\left\{ \mathcal{H}\left( \frac{\pi_{il}}{\pi_{ju}}\right)\right\} + \dfrac{(d-1)(d-2)}{2} (\vert S\vert + 2\g - 2) \right]$$ and
$$d= {\displaystyle \sum_{i=1}^{r}} \riff, \quad \alpha_l = \frac{\alpha_{il}}{\alpha_{jl}}, \pi_{il} = \pi_{i,l_1,\ldots, l_t}.$$
If $\Xi$ is linearly dependent over $\mathbb{C}$, then, there is a subset $W \subseteq \Xi$ such that the elements of $W$ are linearly dependent and no proper subset of $W$ is linearly dependent over $\mathbb{C}$. It follows from Theorem \ref{thm9} that ${\max}\{ \vert n_i \vert\} < C_5$, which contradicts our assumption. Therefore, the set $\Xi$ is linearly independent over $\mathbb{C}$. Applying Lemma \ref{lemh} yields $${\displaystyle \sum_{\nu \in S}^{}}(\nu(G(\nbf)) - {\min}\{ \nu(\pi_{i,l_1,\ldots, l_t}\bsa_i^\nbf)\} \leq \binom{d-1}{2} (\vert S\vert + 2\g - 2),$$ where $S$ is a finite set of place of $L$ containing all zeros and poles of $\alpha_{ij}$ and $\pi_{i,l_1,\ldots, l_t}$ as well as $\mu$ and the infinite places of $L$. Therefore, for a fixed $i$, we have
$$\mu(G(\nbf)) \leq {\max}\{ \mu(\pi_{i,l_1,\ldots, l_t})\}+ \binom{d-1}{2} (\vert S\vert + 2\g - 2)+ {\min}_i \{n_1\mu(\alpha_{i1}) + \cdots + n_t\mu(\alpha_{it})\}.$$ By setting $$C_6 = {\max}\{ \mu(\pi_{i,l_1,\ldots, l_t})\}+ \binom{d-1}{2} (\vert S\vert + 2\g - 2),$$ we conclude 
$$\mu(G_n) \leq C_6 + {\min}_i \{n_1\mu(\alpha_{i1}) + \cdots + n_t\mu(\alpha_{it})\},$$
from which the desired result follows. \hfill$\square$

\section{Proof of Corollary \ref{corf}}\label{sec7}

Before providing a proof of this result, we recall the global derivation over $L$ introduce in \cite{M}. Let $z$ be a transcendental element over $\mathbb{C}$. We denote by $\frac{\partial f}{\partial z}$ the classical differentiation with respect to $z$ of $f \in \mathbb{C}(z)$. This derivation can be extended to a global derivation over $L$ as follows. Let $\Y$ be a primitive element of $L$ over $\mathbb{C}(z)$ and denote by $P(z,Y)$ its minimal polynomial. Then $\Y' := \frac{\partial \Y}{\partial z}$ is defined by $$\Y' = -\frac{\partial P}{\partial z}/ \frac{\partial P}{\partial Y}.$$ Clearly, $\Y'$ is well define since $\frac{\partial P}{\partial Y} \neq 0$ by minimality of $P$. Hence, the mapping $\Y \mapsto \Y'$ on $L$ defines a global derivation over $L$. We write $$P_i(n)= {\displaystyle \sum_{l=0}^{\ri}} a_{il} n^l,$$ where $\ri$ is the degree of $P_i$ and $a_{i0}, \ldots, a_{i\ri} \in L$. Let $z$ be a transcendental element of $L$ and $\beta$ be a primitive element of $L$. By denoting $d= [L:\mathbb{C}(z)]$, it is well known that for every $l=0, \ldots, \ri$, we have $$ a_{il}= {\displaystyle \sum_{k=0}^{d-1}} b_{ilk} \beta^k,$$ with $b_{ilk} \in \mathbb{C}(z)$ and $i \in \{1, \ldots, r\}$. It follows that we may write $$G_n = {\displaystyle \sum_{i=1}^{r} \sum_{l=1}^{r_i}} \pi_{il}q_{il}(n)\alpha^i,$$ where for a fixed $i$ the set $$\{\pi_{il} : l = 1, \ldots, r_i\}$$ is linearly independent over $\mathbb{C}$.
Observe that, $P_i(n)=0$ implies $q_{il}(n)=0$ for all $l \in \{1, \ldots, r_i\}$. It is well-known that there is a computable constant $C_{11}$ such that $n < C_{11}$. 
Put $$\Xi = \{\pi_{il}\alpha_i^n: i= 1,2, \ldots, r, l=1, \ldots, r_i\}$$ and $$q={\displaystyle \sum_{i=1}^{r}}r_i.$$ Using the global differentiation with respect to $z$ above, we define the Wronskian of $\Xi$ as follows
 $$ W(\Xi) = {\rm det} \begin{pmatrix}
\pi_{11}\alpha_1^n & \pi_{12}\alpha_1^n & \cdots & \pi_{rr_r}\alpha_r^n \\
 (\pi_{11}\alpha_1^n)' & (\pi_{12}\alpha_1^n)' & \cdots & (\pi_{rr_r}\alpha_r^n)' \\
\vdots & \vdots & \ddots & \vdots\\
 (\pi_{11}\alpha_1^n)^{(q-1)} & (\pi_{12}\alpha_1^n)^{(q-1)} & \cdots & (\pi_{rr_r}\alpha_t^n)^{(q-1)} \\
\end{pmatrix}.$$
We set $Q_{1,0}(n)=\pi_{11}, Q_{2,0}(n)=\pi_{12}, \ldots,Q_{q,0}(n)=\pi_{rr_r}$, where the indices are ordered lexicographically. For a fixed $i$, we define $$Q_{i,l+1}(x) = Q_{i,l}(x) + x Q_{i, l}'(x) \frac{\alpha_i'}{\alpha_i}$$ for all $l \geq 0$. Now consider the matrix $(Q_{il}(x))$ for $i=1, \ldots, q, l=0, \ldots, q-1$ and its determinant $\Delta(x) \in L[x]$. It is not difficult to see that $\Delta(n)$ equals, up to $${\displaystyle \prod_{i=1}^{r}} (\alpha_i^n)^{r_i},$$ the Wronskian determinant $W(\Xi)$. Therefore, if we assume that $\Xi$ is linearly dependent over $\mathbb{C}$, then $W(\Xi)=0$, which implies $\Delta(n)=0$, since $${\displaystyle \prod_{i=1}^{r}} (\alpha_i^n)^{r_i} \neq 0.$$ By analogy with the previous argument, it follows that there is an effectively computable constant $C_{12}$ such that $n<C_{12}.$ Hence, we conclude that $\Xi$ is linearly independent when $n > {\max}\{C_{11},C_{12}\} =: C_7$. For $n> C_7,$ the set $$\Xi' = \{q_{il}(n)\pi_{il}\alpha_i^n: i= 1,2, \ldots, r, l=1, \ldots, r_i\}$$ is linearly independent over $\mathbb{C}$. Let $S$ be a finite set of places of $L$ containing all zero and pole of $\alpha_i$ for $i=1, \ldots, r$ and of the non-zero $a_{ij}$ for $i=1, \ldots, r$ and $j=1, \ldots, \ri$ as well as $\mu$ and the infinite places of $L$. Proposition \ref{lemh} yields
 $$\mu(G_n) - {\min}\{ \mu(q_{il}(n)\pi_{il}\alpha_i^n)\} \leq \binom{q-1}{2} (\vert S\vert + 2\g - 2)$$ which implies
$$\mu(G_n) \leq C_8 + n \cdot {\min} \{\mu(\alpha_{i})\},$$ where
 $$C_8 = {\max}\{ \mu(\pi_{il})\}+ \binom{q-1}{2} (\vert S\vert + 2\g - 2).$$ This gives us the desired result.
\hfill$\square$

\bibliographystyle{plain}

\begin{thebibliography}{00}

%\bibitem{BV} E. Bombieri and W. Gubler, \emph{Heights in Diophantine Geometry}, Cambridge university press, 2006.

%\bibitem{BW} E. Bombieri and J. Vaaler, \emph{On Siegel's Lemma}, Invent. Math. \textbf{73} (1983), 11--32.

\bibitem{BM} W. D. Brownawell and D. Masser, \emph{Vanishing sums in function fields}. Math. Proc. Cambridge Philos. Soc. \textbf{100} (1986), 427 -- 434. 

\bibitem{EE} B. Edixhoven and J.-H. Evertse, \emph{Diophantine Approximation and Abelian Variety}. Lecture note in Math., Springer Verlag, Berlin, ect., 1993.

\bibitem{E1} J.-H. Evertse, \textit{On sums of $S$-units and linear recurrences}. Comp. Math. \textbf{53(2)} (1984), 225--244.

\bibitem{E} J.-H. Evertse, \textit{An improvement of the quantitative Subspace Theorem}, Compositio Math., \textbf{101}(3) (1996), 225--311.

\bibitem{EG}J-H. Evertse and K. Gy\H{o}ry, \emph{Effective results and methods for Diophantine over finite generated domains}. London Math. Soc. Lecture Notes in Math. \textbf{475}, 2022.

\bibitem{ESS} J.-H. Evertse, H. P. Schlickewei, and W. Schmidt, \textit{Linear equations in variables which lie in a multiplicative group}. Annals of Math. \textbf{155} (2002), 807--836.

\bibitem{FH} C. Fuchs and S. Heintze, \emph{On the growth of linear recurrences in function fields}. Bull. Austr. Math. Soc. \textbf{104(1)} (2021), 11--20. 

\bibitem{FH1} C. Fuchs and S. Heintze, \emph{On the growth of multi-recurrences}. Arch. Math. (Basel) \textbf{119} (2022), 489 -- 494. 

\bibitem{FH2} C. Fuchs and S. Heintze, \emph{A function field variant of Pillai's problem}. J. Number Theory \textbf{222} (2021), 278 -- 292.

\bibitem{FH3} C. Fuchs and S. Heintze, \emph{Integral zeros of a polynomial with linear recurrences as coefficients}. Indag. Math. \textbf{32} (2021), 691 -- 703.

\bibitem{FKK} C. Fuchs, C. Karolus, and D. Kreso, \emph{Decomposable polynomials in second order linear recurrence sequence}. Manuscripta Math. \textbf{159(3)} (2019), 321--346.

\bibitem{FP} C. Fuchs and A. Peth\H{o}, \emph{Effective bounds for the zeros of linear recurrence sequences in function fields}. J. Theor. Nombres Bordeaux \textbf{17(1)} (2005), 749--766. 

\bibitem{K} T. Karimov, E. Kelmendi, J. Nieuwveld, J. Ouaknine, and J. Worrell, \emph{The power of positivity}. In: 2023 38th Annual ACM/IEEE Symposium on Logic in Computer Science (LICS), Boston, MA, USA (2023), 1--11.

\bibitem{L} L. Leroux, \emph{Computing the torsion points of a variety defined by lacunary polynomials }. Math. Comp. \textbf{81(279)} (2012), 1587--1607.

 %\bibitem{LV} B. H. Loxton and and A. J. Van der Poorten, \emph{On the growth of linear recurrences }. Math. Proc. Camb. Phil. Soc. (1977), \textbf{81}, 369.

\bibitem{LV} J. H. Loxton and A. J. Van der Poorten, \emph{On the growth of linear recurrences}. Math. Proc. Camb. Philos. Soc. \textbf{81} (1977), 369--376.

\bibitem{LM} F. Luca, and M. C. Manape, \emph{On the Euler function of linearly recurrence sequences}. Fibonacci Q. \textbf{62}, No. 4, 316--328 (2024).

\bibitem{AN} F. Luca and A. Noubissie, \emph{Linear combinations of factorial and $S$-unit in a ternary recurrence sequence with a double root}. Period. Math. Hung. \textbf{86} (2023), 422--441.

\bibitem{M} R. C. Mason, \emph{Equations over Function Fields }. London Math. Soc. Lecture Notes in Math. \textbf{96}, 1984.

\bibitem{NOW} A. Noubissie, \emph{ Quantitative growth of linear recurrence}. J. Austr. Math. Soc., to appear.

\bibitem{SS} H. P. Schlickewei and W. M. Schmidt , \emph{The number of solution of the polynomials-exponential equations }. Compositio Math. \textbf{120} (2000), 193--225.

%\bibitem{SS} H. P. Schlickewei and W. M. Schmidt, \emph{The number of solutions of Polynomial-Exponential Equations.}. Comp. Math., \textbf{120} (2000), 193--225.

\bibitem{VS} H. P. Schlickewei and A. J. Van der Poorten, \emph{The growth conditions for recurrence sequences}. Macquarie Math. Report \textbf{82-0041} (August 1982), Macquarie University, Australia, 2109. 

\bibitem{S} W. M. Schmidt, \emph{The zero multiplicity of linear recurrence sequences}. Acta Math. \textbf{182} (1999), 243--282.

\bibitem{X} Z. Xiao, \emph{Greatest common divisors for polynomials in almost units and applications to linear recurrence sequences}. Math. Z. \textbf{306(4)}, Paper No. 61, 42 p. (2024).

\bibitem{Z} U. Zannier, \emph{On composite lacunary polynomials and the proof of a conjecture of Schinzel}. Invent. Math. \textbf{174(1)}(2008), 127--138. 

%\bibitem{OW3} J. Ouaknine and J. Worrell . 2014c. \emph{Ultimate Positivity Problem is Decidable for Simple Linear Recurrence Sequences}. In Proc. ICALP (LNCS), Vol. 8573. Springer.

%\bibitem{VS} H. P. Schlickewei and A. J. Van der Poorten, \emph{Zeros of linear recurrence sequences}. Bull. Austral. Math. Soc., \textbf{44} (1991), 215--223.

%\bibitem{SB} W. M. Schmidt, \emph{Diophantine Approximation}. Lecture Notes in Math. \textbf{785}, Springer Verlag, Berlin, ect., 1980. 

%\bibitem{V} A. J. Van der Poorten, \emph{Some problems of recurrent interest}. Topics in classical number Theory, Vol. I, II (Budapest, 1981) Acta Math., \textbf{182} (1999), 1265--1294.

\end{thebibliography}

\end{document}